\documentclass{amsart}
\pdfoutput=1
\usepackage{amssymb}
\usepackage{url}
\usepackage{enumerate}
\usepackage{thmtools}
\usepackage{hyperref} % NOTE: hyperref should generally be the last package loaded

\hypersetup{
  pdftitle={Polynomiality of monotone Hurwitz numbers in higher genera},
  pdfauthor={I. P. Goulden, Mathieu Guay-Paquet, and Jonathan Novak},
}

\declaretheorem[numberwithin=section]{theorem}
\declaretheorem[numberlike=theorem]{lemma}
\declaretheorem[numberlike=theorem]{proposition}

\numberwithin{equation}{section}

\author{I.~P.~Goulden}
\address{Department of Combinatorics \& Optimization\\University of Waterloo\\Canada}
\email{ipgoulden@uwaterloo.ca}

\author{Mathieu Guay-Paquet}
\address{LaCIM\\Universit\'e du Qu\'ebec \`a Montr\'eal\\Canada}
\email{mathieu.guaypaquet@lacim.ca}
\thanks{IPG and MGP were supported by NSERC}

\author{Jonathan Novak}
\address{Department of Mathematics. Massachusetts Institute of Technology, USA}
\email{jnovak@math.mit.edu}

\title{Polynomiality of monotone Hurwitz numbers in higher genera}
\keywords{Hurwitz numbers, matrix models, enumerative geometry}
\subjclass{Primary 05A15, 14E20; Secondary 15B52}

\date{\today}

\newcommand{\ZZ}{\mathbb{Z}}
\newcommand{\QQ}{\mathbb{Q}}

\newcommand{\cm}{{\mathcal{M}}}
\newcommand{\cmbar}{\overline{\cm}}

\newcommand{\pp}{\mathbf{p}}
\newcommand{\qq}{\mathbf{q}}
\newcommand{\rr}{\mathbf{r}}
\newcommand{\xx}{\mathbf{x}}
\newcommand{\yy}{\mathbf{y}}

\newcommand{\hur}{H}
\newcommand{\Hur}{\mathbf{H}}
\newcommand{\mon}{\vec{H}}
\newcommand{\Mon}{\vec{\mathbf{H}}}

\newcommand{\DD}{\mathcal{D}}
\newcommand{\EE}{\mathcal{E}}
\newcommand{\wD}{\widehat{\mathcal{D}}}
\newcommand{\wE}{\widehat{\mathcal{E}}}
\newcommand{\mcR}{\mathcal{R}}

\newcommand{\abs}[1]{\left|{#1}\right|}
\newcommand{\diff}[2][]{\frac{\partial{#1}}{\partial{#2}}}
\newcommand{\sdiff}[3][]{\frac{\partial^2{#1}}{\partial{#2}\partial{#3}}}

\newcommand{\group}[1]{\mathbf{#1}}

\newcommand{\domain}[1]{\mathfrak{#1}}

\newcommand{\function}[1]{\mathcal{#1}}

\newcommand{\Tr}{\operatorname{Tr}}

\DeclareMathOperator{\Aut}{Aut}
\DeclareMathOperator*{\Split}{Split}
\DeclareMathOperator{\trace}{tr}
\DeclareMathOperator{\T}{T}

\let\originalDelta\Delta\renewcommand{\Delta}{\mathop{\originalDelta}\nolimits}
\let\originalPi\Pi\renewcommand{\Pi}{\mathop{\originalPi}\nolimits}

\newcommand{\partref}[1]{\textit{(\ref*{#1})}}
\newcommand{\thmpartref}[2]{\hyperref[#1]{\autoref*{#1}\partref{#2}}}

\begin{document}

\begin{abstract}
Hurwitz numbers count branched covers of the Riemann sphere with specified ramification, or equivalently, transitive permutation factorizations in the symmetric group with specified cycle types. Monotone Hurwitz numbers count a restricted subset of these branched covers, related to the expansion of complete symmetric functions in the Jucys-Murphy elements, and have arisen in recent work on the the asymptotic expansion of the Harish-Chandra-Itzykson-Zuber integral. In previous work we gave an explicit formula for monotone Hurwitz numbers in genus zero. In this paper we consider monotone Hurwitz numbers in higher genera, and prove a number of results that are reminiscent of those for classical Hurwitz numbers. These include an explicit formula for monotone Hurwitz numbers in genus one, and an explicit form for the generating function in arbitrary positive genus. From the form of the generating function we are able to prove that monotone Hurwitz numbers exhibit a polynomiality that is reminiscent of that for the classical Hurwitz numbers, \textit{i.e.}, up to a specified combinatorial factor, the monotone Hurwitz number in genus $g$ with ramification specified by a given partition is a polynomial indexed by $g$ in the parts of the partition.
\end{abstract}

\maketitle
\setcounter{tocdepth}{2}
\tableofcontents

%----------------------------------------------------------------
\section{Introduction}\label{sec:intro}
%----------------------------------------------------------------

%--------------------------------------------------------
\subsection{Classical Hurwitz numbers}
%--------------------------------------------------------

Hurwitz numbers count branched covers of the Riemann sphere with specified ramification data. The most general case which is commonly studied is that of \emph{double} Hurwitz numbers $\hur_g(\alpha, \beta)$, where two points on the sphere are allowed to have non-simple ramification. That is, for two partitions $\alpha, \beta \vdash d$, the Hurwitz number $\hur_g(\alpha, \beta)$ counts degree $d$ branched covers of the Riemann sphere by Riemann surfaces of genus $g$ with ramification type $\alpha$ over 0 (say), ramification type $\beta$ over $\infty$ (say), and simple ramification over $r$ other arbitrary but fixed points (where $r = 2g - 2 + \ell(\alpha) + \ell(\beta)$ by the Riemann-Hurwitz formula), up to isomorphism. The original case of \emph{single} Hurwitz numbers $\hur_g(\alpha)$ is obtained by taking $\beta = (1^d)$, corresponding to having no ramification over $\infty$.

If we label the preimages of some unbranched point by $1, 2, \dotsc, d$, then Hurwitz's monodromy construction~\cite{hurwitz91} identifies $\hur_g(\alpha, \beta)$ bijectively with the number of $(r+2)$-tuples $(\rho, \sigma, \tau_1, \dotsc, \tau_r)$ of permutations in the symmetric group $\group{S}_d$ such that
\begin{enumerate}
  \item
    $\rho$ has cycle type $\alpha$, $\sigma$ has cycle type $\beta$, and the $\tau_i$ are transpositions;
  \item
    the product $\rho \sigma \tau_1 \dotsm \tau_r$ is the identity permutation;
  \item
    the subgroup $\langle \rho, \sigma, \tau_1, \dotsc, \tau_r \rangle \subseteq \group{S}_d$ is transitive; and
  \item
    the number of transpositions is $r = 2g - 2 + \ell(\alpha) + \ell(\beta)$.
\end{enumerate}

The double Hurwitz numbers were first studied by Okounkov~\cite{okounkov00a}, who addressed a conjecture of Pandharipande~\cite{pandharipande00} in Gromov-Witten theory by proving that a certain generating function for these numbers is a solution of the $2$-Toda lattice hierarchy from the theory of integrable systems. Okounkov's result implies that a related generating function for the single Hurwitz numbers is a solution of the KP hierarchy, and as shown by Kazarian and Lando~\cite{kazarian09,kazarian-lando07} via the ELSV formula~\cite{ekedahl-lando-shapiro-vainshtein01}, this further implies the celebrated Witten-Kontsevich theorem~\cite{kontsevich92,witten91} relating intersection theory on moduli spaces to integrable systems. These developments, which revealed rich connections between algebraic geometry and mathematical physics, have led to renewed interest in the Hurwitz enumeration problem.

%--------------------------------------------------------
\subsection{Monotone Hurwitz numbers}
%--------------------------------------------------------

Recently, a new combinatorial twist on Hurwitz numbers emerged in random matrix theory.  Fix a pair $A,B$ of
$N \times N$ normal matrices, and consider the so-called \emph{Harish-Chandra-Itzykson-Zuber integral}
	\begin{equation*}
		\mathcal{I}_N(z;A,B) = \int\limits_{\mathbf{U}(N)} e^{-zN\Tr(AUBU^{-1})} \mathrm{d}U,
	\end{equation*}	
where the integration is over the group of $N \times N$ unitary matrices equipped with its Haar probability measure.
Since $\group{U}(N)$ is compact, the integral converges to define an entire function of the complex variable $z$.
This function is one of the basic special functions of random matrix theory.  A problem of perennial interest,
whose solution would have diverse applications, is
to determine the $N \rightarrow \infty$ asymptotics of $\function{I}_N(z;A_N,B_N)$ when $A_N,B_N$ are given 
sequences of normal matrices which grow in a suitably regular fashion.  

A new approach to the asymptotic analysis of the HCIZ integral was initiated in~\cite{goulden-guay-paquet-novak12c}.
Fix a simply-connected domain $\domain{D}_N$ containing $z=0$ on which $\function{I}_N(z;A,B)$ is non-vanishing.
Then the equation
	\begin{equation*}
		\function{I}_N(z;A,B) = e^{N^2\function{F}_N(z;A,B)}
	\end{equation*}
has a unique holomorphic solution on $\domain{D}_N$ subject to $\function{F}_N(0;A.B)=1$.  In~\cite{goulden-guay-paquet-novak12c},
we proved that, for $1 \leq d \leq N$, the $d$th derivative of $\function{F}_N(z;A,B)$ at $z=0$ is given 
by the absolutely convergent series
	\begin{equation*}
		\function{F}_N^{(d)}(0;A,B) = \sum_{g=0}^{\infty} \frac{C_{g,d}(A,B)}{N^{2g}}
	\end{equation*}
with coefficients
	\begin{equation*}
		  C_{g,d}(A,B) = \sum_{\alpha,\beta \vdash d} (-1)^{d+\ell(\alpha)+\ell(\beta)}
		  \mon_g(\alpha,\beta) \, \frac{p_\alpha(A)}{N^{\ell(\alpha)}} \, \frac{p_\beta(B)}{N^{\ell(\beta)}},
	\end{equation*}
where $p_\alpha(A),p_\beta(B)$ are power-sum symmetric functions specialized at the eigenvalues of $A,B$ and
$\mon_g(\alpha,\beta)$ is the number of $(r+2)$-tuples $(\rho, \sigma, \tau_1, \dotsc, \tau_r)$ of permutations 
  from the symmetric group $\group{S}_d$ such that
  \begin{enumerate}
    \item
      $\rho$ has cycle type $\alpha$, $\sigma$ has cycle type $\beta$, and the $\tau_i$ are transpositions;
    \item
      the product $\rho \sigma \tau_1 \dotsm \tau_r$ is the identity permutation;
    \item
      the subgroup $\langle \rho, \sigma, \tau_1, \dotsc, \tau_r \rangle \subseteq \group{S}_d$ is transitive;
    \item
      the number of transpositions is $r = 2g - 2 + \ell(\alpha) + \ell(\beta)$; and
    \item\label{item:extra-condition}
      writing each $\tau_i$ as $(a_i \, b_i)$ with $a_i < b_i$, we have $b_1 \leq \dotsb \leq b_r$.
  \end{enumerate}

Clearly, if condition (\ref*{item:extra-condition}) is suppressed, the numbers $\mon_g(\alpha,\beta)$ become the classical double Hurwitz numbers $\hur_g(\alpha,\beta)$, so $\mon_g(\alpha,\beta)$ can be seen as counting a restricted subset of the branched covers counted by $\hur_g(\alpha,\beta)$.  
These desymmetrized Hurwitz numbers were dubbed \emph{monotone double Hurwitz numbers} in~\cite{goulden-guay-paquet-novak12c}.

In this paper, we study the monotone \emph{single} Hurwitz numbers $\mon_g(\alpha) = \mon_g(\alpha,1^d)$ and prove 
a monotone analogue of ELSV polynomiality in genus $g \geq 2$.
This result was used in~\cite{goulden-guay-paquet-novak12c} to prove the $N \rightarrow \infty$
convergence of $\function{F}_N(z;A_N,B_N)$ under appropriate hypotheses.  We also obtain an exact formula
for $\mon_1(\alpha)$. Before stating these results, we recall our previous work on monotone Hurwitz numbers in 
genus zero.

%----------------------------------------------------------------
\subsection{Previous results for genus zero}
%----------------------------------------------------------------

We introduce the notational convention $\mon^r(\alpha) = \mon_g(\alpha)$, where it is understood that for a given partition $\alpha \vdash d$, the parameters $r$ and $g$ determine one another via the Riemann-Hurwitz formula $r = 2g - 2 + \ell(\alpha) + d$.

In our previous paper~\cite{goulden-guay-paquet-novak12a} on monotone Hurwitz numbers in genus zero, we considered the generating function for monotone single Hurwitz numbers
\begin{equation}\label{eq:defMon}
  \Mon(z, t, \pp) = \sum_{d\geq 1} \frac{z^d}{d\,!} \sum_{r\geq 0} t^r \sum_{\alpha \vdash d} \mon^r(\alpha) p_\alpha
\end{equation}
as a formal power series in the indeterminates $z, t$ and $\pp = (p_1, p_2, \ldots)$, where $p_\alpha$ denotes the product $\prod_{j=1}^{\ell(\alpha)} p_{\alpha_j}$. We proved the following result, which gives a partial differential equation with initial condition that uniquely determines the generating function $\Mon$. Our proof is a combinatorial join-cut analysis, and we refer to the partial differential equation in this result as the \emph{monotone join-cut equation}.
\begin{theorem}[\cite{goulden-guay-paquet-novak12a}]\label{thm:joincut}
  The generating function $\Mon$ is the unique formal power series solution of the partial differential equation
  \[
    \frac{1}{2t}\left( z\diff[\Mon]{z} - z p_1 \right) = \frac{1}{2} \sum_{i,j \geq 1} \left( (i+j)p_i p_j \diff[\Mon]{p_{i+j}} + ij p_{i+j} \sdiff[\Mon]{p_i}{p_j} + ij p_{i+j} \diff[\Mon]{p_i} \diff[\Mon]{p_j} \right)
  \]
  with the initial condition $[z^0] \Mon = 0$.
\end{theorem}

The differential equation of \autoref{thm:joincut} is the monotone analogue of the classical join-cut equation which determines the single Hurwitz numbers. To make this precise, consider the generating function for the classical single Hurwitz numbers
\begin{equation}
  \Hur(z, t, \pp) = \sum_{d\geq 1}\frac{z^d}{d\,!} \sum_{r\geq 0} \frac{t^r}{r!} \sum_{\alpha \vdash d} \hur^r(\alpha) p_\alpha.
\end{equation}
As shown in\cite{goulden-jackson97,goulden-jackson-vainshtein00}, $\Hur$ is the unique formal power series solution of the partial differential equation (called the  \emph{(classical) join-cut equation})
\begin{equation}\label{eq:classicaljoincut}
  \diff[\Hur]{t} = \frac{1}{2} \sum_{i,j \geq 1} \left( (i+j)p_i p_j \diff[\Hur]{p_{i+j}} + ij p_{i+j} \sdiff[\Hur]{p_i}{p_j} + ij p_{i+j} \diff[\Hur]{p_i} \diff[\Hur]{p_j}\right)
\end{equation}
with the initial condition $[t^0] \Hur = z p_1$. Note that the classical join-cut equation~\eqref{eq:classicaljoincut}  and the monotone join-cut equation given in \autoref{thm:joincut} have exactly the same differential forms on the right-hand side, but differ on the left-hand side, where the differentiated variable is $t$ in one case, and $z$ in the other.

In~\cite{goulden-guay-paquet-novak12a}, we used the monotone join-cut equation to obtain the following explicit formula for the genus zero monotone Hurwitz numbers.

\begin{theorem}[\cite{goulden-guay-paquet-novak12a}]\label{thm:gzformula}
  The genus zero monotone single Hurwitz number $\mon_0(\alpha)$, $\alpha \vdash d$ is given by
  \begin{equation*}
    \mon_0(\alpha) = \frac{d\,!}{\abs{\Aut \alpha}} (2d + 1)^{\overline{\ell(\alpha)-3}} \; \prod_{j=1}^{\ell(\alpha)} \binom{2\alpha_j}{\alpha_j},
  \end{equation*}
  where
  \[
    (2d + 1)^{\overline{k}} = (2d + 1) (2d + 2) \cdots (2d + k)
  \]
  denotes a rising product with $k$ factors, and by convention
  \[
    (2d + 1)^{\overline{k}} = \frac{1}{(2d + k + 1)^{\overline{-k}}},\qquad\qquad k<0.
  \]
\end{theorem}

\autoref{thm:gzformula} is strikingly similar to the well-known explicit formula for the genus zero Hurwitz number
\begin{equation}\label{eq:hur0}
  \hur_0(\alpha) = \frac{d\,!}{\abs{\Aut \alpha}} (d+\ell(\alpha)-2)! \, d^{\,\ell(\alpha)-3} \; \prod_{j=1}^{\ell(\alpha)} \frac{\alpha_j^{\alpha_j}}{\alpha_j!},
\end{equation}
published without proof by Hurwitz~\cite{hurwitz91} in 1891 (see also Strehl~\cite{strehl96}) and independently rediscovered and proved a century later by Goulden and Jackson~\cite{goulden-jackson97}.

%----------------------------------------------------------------
\subsection{Main results}
%----------------------------------------------------------------

In this paper we consider monotone Hurwitz numbers in all positive genera. For genus one, corresponding to branched covers of the sphere by the torus, we obtain the following exact formula.

\begin{theorem}\label{thm:Mongenusone}
  The genus one monotone single Hurwitz number $\mon_1(\alpha)$, $\alpha \vdash d$ is given by
  \begin{multline*}
    \mon_1(\alpha) = \frac{1}{24} \frac{d\,!}{\abs{\Aut \alpha}} \prod_{j=1}^{\ell(\alpha)} \binom{2\alpha_j}{\alpha_j} \\
    \times \left( (2d + 1)^{\overline{\ell(\alpha)}} - 3 (2d + 1)^{\overline{\ell(\alpha) - 1}} - \sum_{k = 2}^{\ell(\alpha)} (k - 2)! (2d + 1)^{\overline{\ell(\alpha) - k}} e_k(2\alpha + 1) \right),
  \end{multline*}
  where $e_k(2\alpha + 1)$ is the $k$th elementary symmetric polynomial of the values $\{2\alpha_i+1 \colon i = 1, 2, \ldots, \ell(\alpha)\}$.
\end{theorem}

For arbitrary genus $g\geq 0$, let
\begin{equation}\label{eqn:genusGFmono}
  \Mon_g(\pp ) = \sum_{d \geq 1} \sum_{\alpha \vdash d} \mon_g(\alpha) \frac{p_\alpha}{d\,!}.
\end{equation}
Our main result, stated below, gives explicit forms for these genus-specific generating functions in all positive genera.

\begin{theorem}\label{thm:main}
  Let $\qq = (q_1, q_2, \ldots)$ be a countable set of formal power series in the indeterminates $\pp = (p_1, p_2, \ldots)$, defined implicitly by the relations
  \begin{equation}\label{eq:pqrel}
    q_j = p_j (1-\gamma )^{-2j}, \qquad j \geq 1,
  \end{equation}
  where $\gamma, \eta, \eta_j$, $j\geq 1$ are formal power series defined by
  \begin{equation*}
    \gamma = \sum_{k \geq 1} \binom{2k}{k} q_k, \quad
    \eta   = \sum_{k \geq 1} (2k + 1) \binom{2k}{k} q_k, \quad
    \eta_j = \sum_{k \geq 1} (2k + 1) k^j \binom{2k}{k} q_k.
  \end{equation*}
  \begin{enumerate}[(i)]
    \item\label{item:genus-one}
      The generating function for genus one monotone Hurwitz numbers is given by
      \[
        \Mon_1 = \tfrac{1}{24} \log\frac{1}{1 - \eta} - \tfrac{1}{8} \log\frac{1}{1 - \gamma}.
      \]
    
    \item\label{item:rational-form}
      For $g \geq 2$, the generating function for genus $g$ monotone Hurwitz numbers is given by
      \[
        \Mon_g = -c_{g,(0)} + \sum_{d = 0}^{3g - 3} \sum_{\alpha \vdash d} \frac{c_{g,\alpha} \, \eta_\alpha}{(1 - \eta)^{\ell(\alpha) + 2g-2}},
      \]
      where the $c_{g,\alpha}$ are rational constants.
    
    \item\label{item:bernoulli}
      For $g \geq 2$, the rational constant $c_{g,(0)}$ is given by
      \[
        c_{g,(0)} = \frac{-B_{2g}}{2g (2g - 2)},
      \]
      where $B_{2g}$ is a Bernoulli number.
  \end{enumerate}
\end{theorem}

Note that our proof of \autoref*{thm:main} is not just an existence proof; the computations to determine the coefficients $c_{g,\alpha}$ are quite feasible in practice if the coefficients for lower values of $g$ are known. For example, for genus $g=2$, we obtain the expression
\begin{equation}\label{eq:MonH2}
6!\cdot\,\vec{\mathbf{H}}_2 = \Big( -3 + \frac{3}{(1 - \eta)^2}\Big) + \frac{5 \eta_3 - 6 \eta_2 - 5 \eta_1}{(1 - \eta)^3} + \frac{29 \eta_1 \eta_2 - 10 \eta_1^2}{(1 - \eta)^4} + \frac{28 \eta_1^3}{(1 - \eta)^5}.
\end{equation}
For genus $g=3$, the corresponding expression for $\Mon_3$ is given in \autoref{sec:forms}.

A key consequence of \autoref{thm:main} is that it implies the polynomiality of the monotone single Hurwitz numbers themselves.

\begin{theorem}\label{thm:polynomiality}
  For each pair $(g,\ell)$ with $(g,\ell) \notin \{(0,1), (0,2)\}$, there is a polynomial $\vec{P}_{g,\ell}$ in $\ell$ variables such that, for all partitions $\alpha \vdash d$, $d\geq 1$, with $\ell$ parts,
  \[
    \mon_g(\alpha) = \frac{d\,!}{\abs{\Aut \alpha}}  \vec{P}_{g,\ell}(\alpha_1, \dots, \alpha_\ell)\; \prod_{j=1}^\ell \binom{2\alpha_j}{\alpha_j}.
  \]
\end{theorem}

%--------------------------------------------------------
\subsection{Comparison with the classical Hurwitz case}
%--------------------------------------------------------

For genus one,  the explicit formula for the monotone Hurwitz number given in \autoref{thm:Mongenusone} is strongly reminiscent of the known formula for the Hurwitz number, given by
\begin{multline*}
  \hur_1(\alpha) = \frac{1}{24} \frac{d\,!}{\abs{\Aut \alpha}} (d + \ell(\alpha))! \prod_{j=1}^{\ell(\alpha)} \frac{\alpha_j^{\alpha_j}}{\alpha_j!} \\
  \times \left( d^{\ell(\alpha)} - d^{\ell(\alpha) - 1} - \sum_{k=2}^{\ell(\alpha)} (k - 2)! d^{\ell(\alpha) - k} e_k(\alpha) \right),
\end{multline*}
which was conjectured in~\cite{goulden-jackson-vainshtein00} and proved by Vakil~\cite{vakil01} (see also~\cite{goulden-jackson99a}).

The expressions for $\Mon_g$ given in \autoref{thm:main} above should be compared with the analogous explicit forms for the generating series
\[
  \Hur_g(\pp) = \sum_{d \geq 1} \sum_{\alpha \vdash d} \frac{\hur_g(\alpha)}{(2g - 2 + \ell(\alpha) + d)!} \frac{p_\alpha }{d\,!}
\]
for the classical Hurwitz numbers. Adapting notation from previous works~\cite{goulden-jackson99a, goulden-jackson99b, goulden-jackson-vakil01} in order to highlight this analogy, let $\rr = (r_1, r_2, \ldots)$ be a countable set of formal power series in the indeterminates $\pp = (p_1, p_2, \ldots)$, defined implicitly by the relations
  \begin{equation}\label{eq:prrel}
    r_j = p_j e^{j\delta}, \qquad j \geq 1,
  \end{equation}
and let $\delta, \phi, \phi_j$, $j\geq 1$ be formal power series defined by
  \begin{equation*}
    \delta = \sum_{k \geq 1} \frac{k^k}{k!} r_k, \qquad
    \phi   = \sum_{k \geq 1} \frac{k^{k+1}}{k!} r_k, \qquad
    \phi_j = \sum_{k \geq 1} \frac{k^{k+j+1}}{k!} r_k.
  \end{equation*}

Then, the genus $g=1$ Hurwitz generating series is~\cite{goulden-jackson99a}
\[
  \Hur_1 = \tfrac{1}{24} \log\frac{1}{1 - \phi} - \tfrac{1}{24}\delta ,
\]
and for $g \geq 2$ we have~\cite{goulden-jackson-vakil01}
\begin{equation}\label{eq:hurwitz-rational}
  \Hur_g =  \sum_{d = 2g-3}^{3g - 3} \sum_{\alpha \vdash d} \frac{a_{g,\alpha} \phi_{\alpha}}{(1 - \phi)^{\ell(\alpha) + 2g-2}},
\end{equation}
where the $a_{g,\alpha}$ are rational constants. For example, when $g=2$ we obtain~\cite{goulden-jackson99b}
\begin{equation}\label{eq:ClaH2}
2^3\cdot 6!\, \mathbf{H}_2 = \frac{5 \phi_3 - 12 \phi_2 + 7 \phi_1}{(1 - \phi)^3} + \frac{29 \phi_1 \phi_2 - 25 \phi_1^2}{(1 - \phi)^4} + \frac{28 \phi_1^3}{(1 - \phi)^5}.
\end{equation}
For genus $g=3$, the corresponding expression for $\Hur_3$ is given in \autoref{sec:forms}.  

\autoref{thm:polynomiality} is the exact analogue of polynomiality for the classical Hurwitz numbers, originally conjectured in~\cite{goulden-jackson-vainshtein00}, which asserts the existence of polynomials $P_{g,\ell}$ such that, for all partitions $\alpha \vdash d$ with $\ell$ parts,
\begin{equation}\label{eqn:ELSV}
  \hur_g(\alpha) = \frac{d\,!}{\abs{\Aut \alpha}} (d + \ell + 2g - 2)! P_{g,\ell}(\alpha_1, \dots, \alpha_\ell) \prod_{j=1}^\ell \frac{\alpha_j^{\alpha_j}}{\alpha_j!}.
\end{equation}

%--------------------------------------------------------
\subsection{A possible geometric interpretation?}
%--------------------------------------------------------

The only known proof of~\autoref{eqn:ELSV} relies on the ELSV formula~\cite{ekedahl-lando-shapiro-vainshtein01},
\begin{equation}\label{eq:elsvf}
  P_{g,\ell}(\alpha_1, \dots, \alpha_\ell) = \int\limits_{\cmbar_{g,\ell}} \frac{1 - \lambda_1 + \cdots + (-1)^g \lambda_g}{(1 - \alpha_1 \psi_1) \cdots (1 - \alpha_\ell \psi_\ell)}.
\end{equation}
Here $\cmbar_{g,\ell}$ is the (compact) moduli space of stable $\ell$-pointed genus $g$ curves, $\psi_1$, $\dots$, $\psi_\ell$ are (complex) codimension 1 classes corresponding to the $\ell$ marked points, and $\lambda_k$ is the (complex codimension $k$) $k$th Chern class of the Hodge bundle. Equation~\eqref{eq:elsvf} should be interpreted as follows: formally invert the denominator as a geometric series; select the terms of codimension $\dim \cmbar_{g,\ell}=3g - 3 + \ell$; and ``intersect'' these terms on $\cmbar_{g,\ell}$. 

In contrast to this, our proof of \autoref{thm:polynomiality} is entirely algebraic and makes no use of geometric methods. A geometric approach to the monotone Hurwitz numbers would be highly desirable. The form of the rational expression given in part~\partref{item:rational-form} of \autoref{thm:main}, in particular its high degree of similarity with the corresponding rational expression~\eqref{eq:hurwitz-rational} for the generating series of the classical Hurwitz numbers, seems to suggest the possibility of an ELSV-type formula for the polynomials $\vec{P}_{g,\ell}$. Further evidence in favour of such a formula is obtained from the values of the rational coefficients that appear in these expressions. First, the Bernoulli numbers have known geometric significance. Second, comparing the expressions~\eqref{eq:MonH2} and~\eqref{eq:ClaH2} for genus $2$ and the expressions in \autoref{sec:forms} for genus $3$ gives strong evidence for the conjecture (now a theorem, see~\cite{guay-paquet12}) that
\begin{equation}\label{eq:power-scaling}
  c_{g,\alpha} = 2^{3g-3} a_{g,\alpha}, \qquad \alpha \vdash 3g-3,
\end{equation}
where $c_{g,\alpha}$ and $a_{g,\alpha}$ are the rational coefficients that appear in \thmpartref{thm:main}{item:rational-form} and~\eqref{eq:hurwitz-rational}, respectively.
But the ELSV formula implies that the coefficients $a_{g,\alpha}$ in the rational form~\eqref{eq:hurwitz-rational} are themselves Hodge integral evaluations, and for the top terms $\alpha \vdash 3g - 3$ these Hodge integrals are free of $\lambda$-classes---the Witten case. Equation~\eqref{eq:power-scaling}, which deals precisely with the case $\alpha \vdash 3g - 3$, might be a good starting point for the formulation of an ELSV-type formula for the monotone Hurwitz numbers.

%----------------------------------------------------------------
\subsection{Organization}
%----------------------------------------------------------------

The bulk of this paper is dedicated to proving parts~\partref{item:genus-one} and~\partref{item:rational-form} of \autoref{thm:main}, which give an explicit expression for the generating function $\Mon_1$ and a rational form for $\Mon_g$, $g \geq 2$. Part~\partref{item:bernoulli} of \autoref{thm:main}, which specifies the lowest order term in the rational form for $\Mon_g$, $g\geq 2$, follows directly from a result of Matsumoto and Novak~\cite{matsumoto-novak10}. For this reason, we present this proof first, in \autoref{sec:bernoulli}.

The necessary definitions and results from our previous paper~\cite{goulden-guay-paquet-novak12a} dealing with the genus zero case are given in \autoref{sec:review}, together with additional technical machinery and results. In \autoref{sec:proofs}, we introduce a particular ring of polynomials, and establish the general form of a transformed version of the generating function $\Mon_g$, $g\geq 1$. In \autoref{sec:gfcns}, we invert this transform, and thus prove parts~\partref{item:genus-one} and~\partref{item:rational-form} of \autoref{thm:main}. In \autoref{sec:explicit}, we use Lagrange's Implicit Function Theorem to evaluate the coefficients in $\Mon_g$, and thus prove \autoref{thm:Mongenusone} and \autoref{thm:polynomiality}.
Finally, the generating functions $\Mon_3$ and $\Hur_3$ are given in \autoref{sec:forms}.

%----------------------------------------------------------------
\subsection{Acknowledgements}
%----------------------------------------------------------------

It is a pleasure to acknowledge helpful conversations with our colleagues Sean Carrell and David Jackson, Waterloo, and Ravi Vakil, Stanford. J.~N.~would like to acknowledge email correspondence with Mike Roth, Queen's. The extensive numerical computations required in this project were performed using Sage~\cite{sage}, and its algebraic combinatorics features developed by the Sage-Combinat community~\cite{sage-combinat}.

%----------------------------------------------------------------
\section{Bernoulli numbers}\label{sec:bernoulli}
%----------------------------------------------------------------

The computation of the constant term $c_{g,(0)}$ for $g \geq 2$ in \autoref{thm:main} relies on a general formula of Matsumoto and Novak (see~\cite{matsumoto-novak10}) for monotone single Hurwitz numbers for the special case of permutations with a single cycle. We give the proof here, as it does not depend on the machinery needed to prove the rest of \autoref{thm:main}.

\begin{proof}[Proof of {\autoref*{thm:main}(\ref*{item:bernoulli})}]
  To compute the monotone single Hurwitz number for a permutation with a single cycle, we can expand the expression for $\Mon_g$ given in \autoref{thm:main} as a power series in $\eta, \eta_1, \eta_2, \ldots$, and then further expand this as a power series in $\pp = (p_1, p_2, \ldots)$, throwing away any terms of degree higher than 1 at each step. For the partition $(d)$ consisting of a single part, this yields the expression
  \begin{align*}
    \mon_g((d))
      &= d\,! [p_d] \Mon_g = d\,! [p_d] \left( (2g - 2) c_{g,(0)} \eta + \sum_{k = 1}^{3g - 3} c_{g,(k)} \eta_k \right) \\
      &= \frac{(2d)!}{d\,!} \left( (2g - 2) c_{g,(0)} \, (2d + 1) + \sum_{k = 1}^{3g - 3} c_{g,(k)} (2d + 1) d^k \right).
  \end{align*}
  For fixed $g$, this expression is $(2d)! / d\,!$ times a polynomial in $d$, and evaluating this polynomial at $d = 0$ gives $(2g - 2) c_{g,(0)}$. In contrast, according to Matsumoto and Novak's formula~\cite[Equation (48)]{matsumoto-novak10}, we have
  \[
    \mon_g((d)) = \frac{(2d)!}{d\,!} \binom{2g - 2 + 2d}{2g - 2} \frac{1}{2g (2g - 1)} \left[\frac{z^{2g}}{(2g)!}\right] \left( \frac{\sinh(z/2)}{z/2} \right)^{2d - 2}.
  \]
  Again, for fixed $g$, this expression is $(2d)! / d\,!$ times a polynomial in $d$. Evaluating this polynomial at $d = 0$ gives
  \begin{align*}
    (2g - 2) c_{g,(0)}
      &= \frac{1}{2g (2g - 1)} \left[\frac{z^{2g}}{(2g)!}\right] \left( \frac{\sinh(z/2)}{z/2} \right)^{-2} \\
      &= \frac{1}{2g (2g - 1)} \left[\frac{z^{2g}}{(2g)!}\right] \left( z\diff{z} - 1 \right) \frac{-z}{e^z - 1} = \frac{-B_{2g}}{2g},
  \end{align*}
  since $z / (e^z - 1)$ is the exponential generating function for the Bernoulli numbers.
\end{proof}

%----------------------------------------------------------------
\section{Algebraic methodology and a change of variables}\label{sec:review}
%----------------------------------------------------------------

%----------------------------------------------------------------
\subsection{Algebraic methodology}\label{sec:algmeth}
%----------------------------------------------------------------

In our previous paper~\cite{goulden-guay-paquet-novak12a} on monotone Hurwitz numbers in genus zero, we introduced three (families of) operators: the \emph{lifting} operators $\Delta_i$, the \emph{projection} operators $\Pi_i$, and the \emph{splitting} operators $\Split_{i \to j}$, which involve the indeterminates $\pp =  (p_1, p_2, \ldots)$ and a collection of auxiliary indeterminates $\xx = (x_1, x_2, \ldots)$. These operators were defined by
\begin{align*}
&\Delta_i = \sum_{k \geq 1} k x_i^k \diff{p_k}, \\
&\Pi_i = [x_i^0] + \sum_{k \geq 1} p_k [x_i^k], \\
&\Split_{i \to j} F(x_i) = \frac{x_j F(x_i) - x_i F(x_j)}{x_i - x_j} + F(0).
\end{align*}

In terms of these operators, the genus-specific generating functions $\Mon_g$ defined in~\autoref{eqn:genusGFmono} for $g\geq 0$ are characterized by the following result, which is essentially a reworking of the monotone join-cut equation given in \autoref{thm:joincut}.
\begin{theorem}[{\cite{goulden-guay-paquet-novak12a}}]\label{thm:pdeunique}
  \ \par % NOTE: without something like this, references to "thm:pdeunique" do not work correctly.
  \begin{enumerate}[(i)]
    \item\label{item:pdeunique-genus-zero}
      The generating function $\Delta_1 \Mon_0$ is the unique formal power series solution in the ring $\QQ[[\pp, x_1]]$ of
      \begin{equation*}
        \Delta_1 \Mon_0 = \Pi_2 \Split_{1 \to 2} \Delta_1 \Mon_0 + (\Delta_1 \Mon_0)^2 + x_1
      \end{equation*}
      with the initial condition $[p_0 x_1^0] \Delta_1 \Mon_0 = 0$.
    
    \item\label{item:pdeunique-higher-genera}
      For $g \geq 1$, $\Delta_1 \Mon_g$ is uniquely determined in terms of $\Delta_1 \Mon_i$, $0\leq i\leq g-1$, by
      \[
        \left( 1 - 2 \Delta_1 \Mon_0 - \Pi_2 \Split_{1 \to 2} \right) \Delta_1 \Mon_g = \Delta_1^2 \Mon_{g-1} + \sum_{g'=1}^{g-1} \Delta_1 \Mon_{g'} \, \Delta_1 \Mon_{g-g'}.
      \]
    
    \item
      For $g \geq 0$, the generating function $\Mon_g$ is uniquely determined by the generating function $\Delta_1 \Mon_g$ and the fact that $[p_0] \Mon_g = 0$.
  \end{enumerate}
\end{theorem}

%----------------------------------------------------------------
\subsection{A change of variables}\label{sec:changevars}
%----------------------------------------------------------------

In~\cite{goulden-guay-paquet-novak12a}, where we determined $\Delta_1\Mon_0$ from \thmpartref{thm:pdeunique}{item:pdeunique-genus-zero}, we found it convenient to change variables from $\pp = (p_1, p_2, \ldots)$ and $\xx = (x_1, x_2, \ldots)$ to $\qq = (q_1, q_2, \ldots)$ and $\yy = (y_1, y_2, \ldots)$ via the relations
\begin{equation}\label{eq:def_change}
  q_j = p_j (1 - \gamma)^{-2j}, \qquad y_j = x_j (1 - \gamma)^{-2}, \qquad j \geq 1,
\end{equation}
and to define the formal power series $\gamma, \eta, \eta_j$, $j\geq 1$ by
  \begin{equation*}
    \gamma = \sum_{k \geq 1} \binom{2k}{k} q_k, \quad
    \eta   = \sum_{k \geq 1} (2k + 1) \binom{2k}{k} q_k, \quad
    \eta_j = \sum_{k \geq 1} (2k + 1) k^j \binom{2k}{k} q_k.
  \end{equation*}
Expressing the operators $\Delta_i$, $\Pi_i$, $\Split_{i \to j}$ in terms of $\qq$ and $\yy$, we obtained
\begin{align*}
&\Delta_i = \sum_{k \geq 1} \left( k y_i^k \diff{q_k} \right) + \frac{4y_i}{(1 - 4y_i)^{\frac{3}{2}} (1 - \eta)} \sum_{k \geq 1} \left( k q_k \diff{q_k} + y_k \diff{y_k} \right),\\
&\Pi_i  = [y_i^0]  + \sum_{k \geq 1} q_k [y_i^k], \\
&\Split_{i \to j} F(y_i) =\frac{y_j F(y_i) - y_i F(y_j)}{y_i - y_j} + F(0).
\end{align*}
We were also able to show that
\begin{equation}\label{eq:dieqdiep}
\EE=\frac{1-\eta}{1-\gamma}\DD ,
\end{equation}
where $\DD =\sum_{k\geq 1}k p_k \diff{p_k}$, and $\EE =\sum_{k\geq 1}k q_k \diff{q_k}$, and, for each $k\geq 1$, that
\[
q_k\diff{q_k}=p_k\diff{p_k}-\frac{2q_k}{1-\gamma}\binom{2k}{k}\DD.
\]
Summing this over $k\geq 1$ gives
\begin{equation}\label{eq:hatdieqdiep}
\wE=\wD-\frac{2\gamma}{1-\gamma}\DD,
\end{equation}
where $\wD =\sum_{k\geq 1} p_k \diff{p_k}$, and $\wE =\sum_{k\geq 1} q_k \diff{q_k}$.

In terms of these transformed variables, we were able to solve the monotone join-cut equation for genus $0$ given in \thmpartref{thm:pdeunique}{item:pdeunique-genus-zero}, to obtain \cite[Corollary 4.3]{goulden-guay-paquet-novak12a}
\begin{equation}\label{eq:delta_zero}
  \Delta_1 \Mon_0 = \Pi_2 A,\hspace{1.8em} A=  1 - (1 - 4y_1)^{\frac{1}{2}} - \frac{y_1 (1 - 4y_1)^{\frac{1}{2}}}{2(y_1 - y_2)} \left( (1 - 4y_1)^{-\frac{1}{2}} - (1 - 4y_2)^{-\frac{1}{2}}  \right).
\end{equation}

In this paper we will be solving the monotone join-cut equation for genus $g$ given in \thmpartref{thm:pdeunique}{item:pdeunique-higher-genera}. The following result will allow us to reexpress the left-hand side of this equation in a more tractable form.
  \begin{proposition}\label{thm:lhs_T}
    For $g \geq 1$, we have
    \[  \left( 1 - 2 \Delta_1 \Mon_0 - \Pi_2 \Split_{1 \to 2} \right) \Delta_1 \Mon_g 
        = (1 - \T) \left( (1 - \eta) (1 - 4y_1)^{\frac{1}{2}} \Delta_1 \Mon_g \right), \]
    where $\T$ is the $\QQ[[\qq]]$-linear operator defined by
    \[
      \T(F) = (1 - \eta)^{-1} \Pi_2 (1 - 4y_2)^{-\frac{3}{2}} \Split_{1 \to 2} \left( (1 - 4y_1) F \right).
    \]
  \end{proposition}
  
  \begin{proof}
From~\eqref{eq:delta_zero} and the expression for $\Delta_1$ given above (and using the fact that $\Delta_1 \Mon_g$ has no constant term as a power series in $y_1$), we have
\begin{equation*}
\begin{aligned}
\text{LHS} &=\Pi_2 \left(\left(1-2A\right)\Delta_1\Mon_g-\frac{y_2\Delta_1\Mon_g-y_1\Delta_2\Mon_g}{y_1-y_2}\right)\\
           &= \Pi_2\left(  \frac{\left(  (y_1-y_2)(1-2A)-y_2\right)\Delta_1\Mon_g+y_1\Delta_2\Mon_g}{y_1-y_2} \right) .
\end{aligned}
\end{equation*}
But it is routine to check that
\begin{multline*}
 (y_1-y_2)(1-2A)-y_2 \\
= (y_1-y_2)\left(2-(1-4y_2)^{-\frac{3}{2}}\right)(1-4y_1)^{\frac{1}{2}}-y_2(1-4y_2)^{-\frac{3}{2}}(1-4y_1)^{\frac{3}{2}},
\end{multline*}
so we have
\begin{equation*}
\begin{aligned}
\text{LHS} &= \Pi_2\Bigg( \left(2-(1-4y_2)^{-\frac{3}{2}}\right)(1-4y_1)^{\frac{1}{2}}\Delta_1\Mon_g \\
&\qquad\qquad\qquad\qquad\qquad\qquad -\frac{y_2(1-4y_2)^{-\frac{3}{2}}(1-4y_1)^{\frac{3}{2}}\Delta_1\Mon_g-y_1\Delta_2\Mon_g}{y_1-y_2}\Bigg) \\
           &= (1-\eta)(1-4y_1)^{\frac{1}{2}}\Delta_1\Mon_g - \Pi_2(1 - 4y_2)^{-\frac{3}{2}} \Split_{1 \to 2} \left( (1 - 4y_1)^{\frac{3}{2}} \Delta_1 \Mon_g \right) ,
\end{aligned}
\end{equation*}
giving the result.
\end{proof}

%----------------------------------------------------------------
\subsection{Auxiliary power series}\label{sec:auxseries}
%----------------------------------------------------------------

We also find it convenient to define auxiliary power series related to the power series $\gamma$, $\eta$ and $\eta_j$, $j \geq 1$ in $\QQ[[\qq]]$ which appear in the statement of \autoref{thm:main}. These are the power series $\gamma(y_i)$, $\eta(y_i)$, and $\eta_j(y_i)$, $j \geq 1$ in $\QQ[[\yy]]$, defined by
\begin{align*}
  \gamma(y_i) &= (1 - 4y_i)^{-\frac{1}{2}} - 1 = \sum_{k \geq 1} \binom{2k}{k} y_i^k, \\
  \eta(y_i)   &= (1 - 4y_i)^{-\frac{3}{2}} - 1 = \sum_{k \geq 1} (2k + 1) \binom{2k}{k} y_i^k, \\
  \eta_j(y_i) &= \left( y_i \diff{y_i} \right)^j (1 - 4y_i)^{-\frac{3}{2}} = \sum_{k \geq 1} (2k + 1) k^j \binom{2k}{k} y_i^k, \qquad j \geq 1,
\end{align*}
so that
\begin{equation}\label{eq:PiAux}
  \Pi_i \gamma(y_i) = \gamma, \qquad
  \Pi_i \eta(y_i) = \eta, \qquad
  \Pi_i \eta_j(y_i) = \eta_j, \qquad j \geq 1.
\end{equation}

%----------------------------------------------------------------
\subsection{Computational lemmas}
%----------------------------------------------------------------

The following computational lemmas are used extensively in the rest of the paper to apply the lifting operator $\Delta_1$ to expressions involving the indeterminates $\yy$ and the series $\gamma, \eta, \eta_1, \eta_2, \ldots$.

\begin{lemma}\label{thm:delta_pi}
  For $F \in \QQ[[\qq, \yy]]$, we have the identity
  \[
    \Delta_1 \Pi_2 F = \Pi_2 \Delta_1 F + y_1 \left. \diff[F]{y_2} \right|_{y_2 = y_1}.
  \]
\end{lemma}

\begin{proof}
We compute directly the commutator
\begin{equation*}
\begin{aligned}
 \Delta_1 \Pi_2  -  \Pi_2 \Delta_1 &= \sum_{k\geq 1} \left(  ky_1^k[y_2^k]+  \frac{4y_1}{(1 - 4y_1)^{\frac{3}{2}} (1 - \eta)} \left(kq_k[y_2^k]-q_k [y_2^k]y_2\diff{y_2} \right)   \right) \\
 &=\sum_{k\geq 1} ky_1^k[y_2^k],
\end{aligned}
\end{equation*}
and the result follows immediately.
\end{proof}

\begin{lemma}\label{thm:delta_values}
  We have
  \[\begin{aligned}
    &\Delta_1 y_j = 4 y_1 y_j (1 - 4y_1)^{-\frac{3}{2}} (1 - \eta)^{-1}, \qquad j \geq 1, \\
    &\Delta_1 \eta = \eta_1(y_1) + 4 y_1 (1 - 4y_1)^{-\frac{3}{2}} \eta_1 (1 - \eta)^{-1}, \\
    &\Delta_1 \eta_j = \eta_{j + 1}(y_1) + 4 y_1 (1 - 4y_1)^{-\frac{3}{2}} \eta_{j + 1} (1 - \eta)^{-1}, \qquad j \geq 1. \\
  \end{aligned}\]
\end{lemma}

\begin{proof}
  The first equation follows directly from the expression for $\Delta_1$ given in \autoref{sec:changevars}. The other two equations can be obtained by applying \autoref{thm:delta_pi} to the expressions in~\eqref{eq:PiAux} for $\eta$ and $\eta_j$.
  \end{proof}

\begin{proposition}\label{prop:DelSqMon}
\[
\Delta_1^2 \Mon_0 = y_1^2 (1 - 4y_1)^{-2}
\]
\end{proposition}

\begin{proof}
Using~\eqref{eq:delta_zero}  and \autoref{thm:delta_pi}, we obtain
\begin{equation}\label{eq:DsqMon0}
\Delta_1^2 \Mon_0 =\Pi_2\Delta_1 A +y_1 \left. \diff[A]{y_2} \right|_{y_2 = y_1}.
\end{equation}

We will consider the two terms in \eqref{eq:DsqMon0} separately. For the first term, from $\Delta_1 \frac{y_1}{y_1-y_2}=0$ (since $\frac{y_1}{y_1-y_2} = \frac{x_1}{x_1-x_2}$ and $\Delta_1(x_1) = \Delta_1(x_2) = 0$) and of course $\Delta_1(1)=0$), we have
\begin{equation*}
\begin{aligned}
\Delta_1 A =&\left( -1+\frac{y_1}{2(y_1-y_2)}(1-4y_2)^{-\frac{1}{2}}\right)\Delta_1(1-4y_1)^{\frac{1}{2}}\\
                                 &+\frac{y_1}{2(y_1-y_2)}(1-4y_1)^{\frac{1}{2}}\Delta_1(1-4y_2)^{-\frac{1}{2}}.
\end{aligned}
\end{equation*}
Applying \autoref{thm:delta_values}, it is now routine to show that
\[
\Delta_1 A =4y_1^2(1-4y_1)^{-2}(1-\eta)^{-1}\left(2-(1-4y_2)^{-\frac{3}{2}}\right),
\]
and so we obtain
\[
\Pi_2 \Delta_1 A =4y_1^2(1-4y_1)^{-2}.
\]

For the second term, we have
\begin{equation*}
A=1 - (1-4y_1)^{\frac{1}{2}} - \frac{1}{2} y_1(1-4y_1)^{\frac{1}{2}}\sum_{k\geq 1} \binom{2k}{k}\sum_{i=0}^{k-1} y_1^{k-1-i} y_2^i,
\end{equation*}
which gives
\begin{equation*}
\begin{aligned}
y_1 \left. \diff[A]{y_2} \right|_{y_2 = y_1} &= -\frac{1}{2} y_1^2 (1-4y_1)^{\frac{1}{2}}\sum_{k\geq 2} \binom{2k}{k}\binom{k}{2} y_1^{k-2}\\
& = -\frac{1}{4} y_1^2 (1-4y_1)^{\frac{1}{2}} \frac{\partial^2}{\partial y_1^2} (1-4y_1)^{-\frac{1}{2}} = -3 y_1^2 (1-4y_1)^{-2} .
\end{aligned}
\end{equation*}

The result follows immediately from~\eqref{eq:DsqMon0} by combining these two terms.
\end{proof}
 
%----------------------------------------------------------------
\section{A ring of polynomials and solving the join-cut equation}\label{sec:proofs}
%----------------------------------------------------------------

In this section we consider the monotone join-cut equation for $\Delta_1\Mon_g$ given in \thmpartref{thm:pdeunique}{item:pdeunique-higher-genera}, with the differential operator on the left-hand side reexpressed in the form given in \autoref{thm:lhs_T}, to give
\begin{equation}\label{eq:revjcut}
(1 - \T) \left( (1 - \eta) (1 - 4y_1)^{\frac{1}{2}} \Delta_1 \Mon_g \right)
=\Delta_1^2 \Mon_{g-1} + \sum_{g'=1}^{g-1} \Delta_1 \Mon_{g'} \, \Delta_1 \Mon_{g-g'} .
\end{equation}
In order to determine the form of the solution $\Delta_1\Mon_g$ for $g\geq 1$, we will find it convenient to work in the
ring $\mcR$ of polynomials in $(1-4y_1)^{-1}$ and $\{ \eta_k (1 - \eta)^{-1}\}_{k \geq 1}$ over $\QQ$.  For $r\in\mcR$, the \emph{weighted degree} of $r$ is its degree as a polynomial in these quantities, where $(1-4y_1)^{-1}$ has degree 1, and $\eta_k (1 - \eta)^{-1}$ has degree $k$, $k\geq 1$. For $d\geq 0$, we let $\mcR_d$ denote the set of polynomials in $\mcR$ whose weighted degree is at most $d$.
  
As an example of this notation, from~\eqref{eq:PiAux} we immediately deduce that
\begin{equation}\label{eq:PiAuxR}
(1-\eta)^{-1}\Pi_2\left( y_2^k (1-4y_2)^{-\frac{3}{2}-k} \right) \in \mcR_{k}, \qquad k\geq 1.
\end{equation}

  \begin{proposition}\label{thm:TonR}
The operator $\T$ sends the ring $\mcR$ to itself, is locally nilpotent and preserves weighted degrees. The operator $1-T$ preserves weighted degrees and is invertible on $\mcR$.
  \end{proposition}
  
  \begin{proof}
Let $Y_i=y_i(1-4y_i)^{-1}$, $i=1,2$. Since $\T(\eta_k(1-\eta)^{-1}r) =\eta_k(1-\eta)^{-1}\T(r)$ for $k\geq 1$ and $r\in\mcR$, it is sufficient to prove the result for the basis $\{ Y_1^k\}_{k\geq 0}$, in which $Y_1^k$ has weighted degree $k$. For $k = 0$, we have $T(1)=0$. For $k \geq 1$, it is routine to check that
\[
\Split_{1 \to 2} \left( (1 - 4y_1) Y_1^k\right)
= \frac{Y_1^kY_2 -Y_1 Y_2^k}{Y_1-Y_2}
\]
which equals $0$ for $k=1$. Thus $\T(Y_1)=0$, and for $k\geq 2$, we have
\begin{equation}\label{eq:TY1k}
      \T\left( Y_1^k \right)
=\sum_{i=1}^{k-1}  Y_1^{k-i}(1-\eta)^{-1}\Pi_2\left( (1-4y_2)^{-\frac{3}{2}} Y_2^{i}\right) .
\end{equation}
But $ Y_1^{k-i}\in\mcR_{k-i}$, and $k-i<k$ for all $i=1,\ldots ,k-1$. Also, from~\eqref{eq:PiAuxR} we have
\[
(1-\eta)^{-1}\Pi_2\left( (1-4y_2)^{-\frac{3}{2}}Y_2^{i}\right)   =(1-\eta)^{-1}\Pi_2\left( y_2^i(1-4y_2)^{-\frac{3}{2}-i}   \right)\in\mcR_i,
\]
which implies that $\T\left( Y_1^k \right) \in \mcR_k$. Furthermore, its degree in $(1-4y_1)^{-1}$ is strictly less than $k$, and, since $\T(1)=\T(Y_1)=0$, it follows that repeated application of $\T$ to any element of $\mcR$ is eventually zero.

Of course, the operator $1-\T$ also preserves weighted degrees, and it is invertible, with inverse given for any $r\in\mcR_d$ by
\[
(1-\T)^{-1}r=(1+\T+\T^2+\ldots)r = (1+\T+\ldots +\T^{d-1})r,
\]
since, from the proof above, $\T^i(r)=0$ for any $i\geq d$.
  \end{proof}
  
  \begin{proposition}\label{prop:DelR}
  For $r\in\mcR_d$ and $m \in \ZZ$, we have
  \begin{enumerate}[(i)]
  \item\label{item:DelR-i}
  \[
  (1-\eta)^{m+1} (1-4y_1)^{\frac{1}{2}}\Delta_1\left(  (1-\eta)^{-m}  r\right)\in\mcR_{d+2},
  \]
  \item\label{item:DelR-ii}
  \[
  (1-\eta)^{m+1} \Delta_1\left(  (1-\eta)^{-m}(1-4y_1)^{-\frac{1}{2}}  r\right)\in\mcR_{d+3}.
  \]
  \end{enumerate}
  \end{proposition}
  
  \begin{proof}
 Since $\Delta_1$ is a linear differential operator, it is sufficient to prove these results for a generic monomial $\mu=(1-4y_1)^{-k}\eta_{b_1}\cdots\eta_{b_j}(1-\eta)^{-j}$ with $k+b_1+\ldots +b_j=d$. Let $\rho=(1-\eta)^{-m} \mu$. 

For part~\partref{item:DelR-i}, apply the product rule to obtain 
\[
\Delta_1 \rho = 4k(1-4y_1)^{-1}\rho\Delta_1 y_1+(j+m)(1-\eta)^{-1}\rho\Delta_1\eta + \sum_{i=1}^j \frac{\rho}{\eta_{b_i}}\Delta_1\eta_{b_i} .
\]
Multiplying this equation by $(1-\eta)^{m+1} (1-4y_1)^{\frac{1}{2}}$ and applying  \autoref{thm:delta_values} and~\eqref{eq:PiAuxR}, it is straightforward to prove that each of the $j+2$ terms is contained in $\mcR_{d+2}$, giving the result.

For part~\partref{item:DelR-ii}, apply the product rule to determine $\Delta_1\left( (1-4y_1)^{-\frac{1}{2}}\rho \right)$, and the result follows from part~\partref{item:DelR-i} and \autoref{thm:delta_values}. 
\end{proof}
  
  We are now able to give an explicit form for $\Delta_1\Mon_g$, for any positive choice of genus $g$.
  
\begin{theorem}\label{thm:delta_H}
For $g \geq 1$,
  \[
   (1 - \eta)^{2g - 1} (1 - 4y_1)^{\frac{1}{2}} \Delta_1 \Mon_g \in\mcR_{3g-1}.
  \]
\end{theorem}

\begin{proof}
We proceed by induction on $g$. For the base case $g=1$, equation~\eqref{eq:revjcut} and \autoref{prop:DelSqMon} give
\begin{equation}\label{eq:jcutg1}
(1-T)(1-\eta)(1-4y_1)^{\frac{1}{2}}\Delta_1\Mon_1=y_1^2 (1 - 4y_1)^{-2}=Y_1^2\in\mcR_2 ,
\end{equation} 
and the result for $g=1$ follows immediately from \autoref{thm:TonR}.
   
Now consider an arbitrary $g \geq 2$, with the induction hypothesis that the result holds for all smaller positive values. Then if we multiply~\eqref{eq:revjcut} by $(1-\eta)^{2g-2}$, we obtain the equation
\begin{multline}\label{eq:indHg}
  (1-T)(1-\eta)^{2g-1}(1-4y_1)^{\frac{1}{2}}\Delta_1\Mon_g \\
  =  (1 - \eta)^{2g - 2} \Delta_1^2 \Mon_{g-1} + \sum_{g'=1}^{g-1} (1 - \eta)^{2g - 2}\Delta_1 \Mon_{g'} \, \Delta_1 \Mon_{g-g'}.
\end{multline}
Now consider the terms on the right-hand side of~\eqref{eq:indHg}. The term corresponding to the summand $g'$ can be written as
\[
(1-4y_1)^{-1}\left( (1 - \eta)^{2g' - 1} (1 - 4y_1)^{\frac{1}{2}} \Delta_1 \Mon_{g'}\right) \left( (1 - \eta)^{2(g-g') - 1} (1 - 4y_1)^{\frac{1}{2}}\Delta_1\Mon_{g-g'}\right)
\]
and from the induction hypothesis, this has weighted degree at most $1+(3g'-1)+(3(g-g')-1) = 3g-1$. For the remaining term, first apply the induction hypothesis to give
\[
\Delta_1\Mon_{g-1}=(1-\eta)^{3-2g}(1-4y_1)^{-\frac{1}{2}}r,\qquad \text{where}\quad r\in\mcR_{3g-4}.
\]
Then from \thmpartref{prop:DelR}{item:DelR-ii}, we have $(1-\eta)^{2g-2} \Delta_1^2\Mon_{g-1}\in\mcR_{(3g-4)+3}$. Thus all terms on the right-hand side of~\eqref{eq:indHg} have weighted degree at most $3g-1$. The result for $g$ follows immediately from \autoref{thm:TonR}.
  \end{proof}

%----------------------------------------------------------------
\section{Generating functions for monotone Hurwitz numbers}\label{sec:gfcns}
%----------------------------------------------------------------

In the previous section, we obtained results for $\Delta_1\Mon_g$, $g\geq 1$. In this section, we consider how to invert the operator $\Delta_1$, in order to obtain results for the generating function $\Mon_g$ itself, and thus prove parts~\partref{item:genus-one} and~\partref{item:rational-form} of \autoref{thm:main}. To accomplish this, we introduce the operator $\Theta_t$, whose action on elements of $\QQ[[\qq]]$ is the substitution $q_j \mapsto q_j t$, $j\geq 1$. For example, we immediately have
\begin{equation}\label{eq:propThetaseries}
\Theta_t\gamma=\gamma t,\qquad \Theta_t\eta=\eta t,\qquad \Theta_t\eta_j=\eta_j t, \quad j\geq 1,
\end{equation}
and for the operator $\wE$ introduced in~\eqref{eq:hatdieqdiep}, we have
\begin{equation}\label{eq:ThetawE}
\Theta_t\wE=t\diff{t}\Theta_t.
\end{equation}
Since
\begin{align*}
  \DD &= \sum_{k \geq 1} k p_k \diff{p_k} = \Pi_1 \Delta_1, \\
  \wD &= \sum_{k \geq 1} p_k \diff{p_k} = \Pi_1 \int\limits_0^{x_1} \frac{\mathrm{d}x_1 \Delta_1}{x_1} = \Pi_1 \int\limits_0^{y_1} \frac{\mathrm{d}y_1 \Delta_1}{y_1},
\end{align*}
we can apply $\Theta_t$ to equation~\eqref{eq:hatdieqdiep} to obtain
\[
\Theta_t\wE=\Theta_t\Phi\Delta_1,
\]
where 
\begin{equation}\label{eq:opPhi}
\Phi=\Pi_1\left( \int\limits_0^{y_1}\frac{\mathrm{d}y_1}{y_1}-\frac{2\gamma}{1-\gamma}\right) . 
\end{equation}
Thus, applying~\eqref{eq:ThetawE} to $\Mon_g$, we obtain
\begin{equation}\label{eq:invertDel}
\Mon_g=\int\limits_{0}^{1}\frac{\mathrm{d}t}{t}\Theta_t\Phi\Delta_1\Mon_g,\qquad g\geq 1.
\end{equation}

For the operator $\Phi$, using~\eqref{eq:PiAux}, it is straightforward to check that, for $j\geq 2$,
\begin{equation}\label{eq:propPhi}
\Phi (\eta(y_1)-\gamma(y_1))=\frac{2\gamma(1-\eta)}{1-\gamma},\;\; \Phi\eta_1(y_1)=\eta-\frac{2\gamma}{1-\gamma}\eta_1,\;\;\Phi\eta_j(y_1)=\eta_{j-1}-\frac{2\gamma}{1-\gamma}\eta_j .
\end{equation}

We are now able to deduce the explicit expression for the genus one monotone Hurwitz generating function stated in \thmpartref{thm:main}{item:genus-one}.

\begin{theorem}\label{cor:genfcnone}
  The generating function for genus one monotone Hurwitz numbers is given by
  \[
    \Mon_1 = \tfrac{1}{24} \log\frac{1}{1 - \eta} - \tfrac{1}{8} \log\frac{1}{1 - \gamma}.
  \]
\end{theorem}

\begin{proof}
  
From \autoref{eq:jcutg1} and \autoref{thm:TonR}, we have
\[
\Delta_1\Mon_1=(1-\eta)^{-1}(1-4y_1)^{-\frac{1}{2}}(1+\T)Y_1^2,
\]
and, simplifying this using~\eqref{eq:TY1k} with $k=2$, after noting that $4y_1 (1 - 4y_1)^{-\frac{3}{2}}=\eta(y_1) - \gamma(y_1)$, we obtain
  \begin{equation*}
  \begin{aligned}
   \Theta_t \Phi \Delta_1 \Mon_1 &=\Theta_t \Phi \left(\frac{2\eta_1(y_1) - 3\eta(y_1) + 3\gamma(y_1)}{48(1 - \eta)} + \frac{(\eta(y_1) - \gamma(y_1)) \eta_1}{24(1 - \eta)^2}\right) \\
   &= \frac{\eta t}{24(1-\eta t)} - \frac{\gamma t}{8(1-\gamma t)},
   \end{aligned}
  \end{equation*}
where for the second equality we have used~\eqref{eq:propPhi} and simplified, and then
 applied~\eqref{eq:propThetaseries}.
The result follows from~\eqref{eq:invertDel}, together with the fact that $\Mon_1$ has constant term $0$.
\end{proof}

For genus two or more, we are able to obtain a polynomiality result for the monotone Hurwitz number generating function $\Mon_g$.

\begin{theorem}\label{thm:MonTwoPoly}
  For $g \geq 2$, we have
  \[
  \Mon_g\in\QQ [\{ \eta_k(1-\eta)^{-1}\}_{k\geq 1}, (1-\eta)^{-1}].
  \]
  Moreover, each monomial $\eta_{\alpha}(1-\eta)^{-\ell(\alpha)-n}$ that appears in $\Mon_g$ has weighted degree $\abs{\alpha} \leq 3g-3$ in $\{ \eta_k(1-\eta)^{-1}\}_{k\geq 1}$, and degree  $n\leq 2g-2$ in $(1-\eta)^{-1}$.
\end{theorem}

\begin{proof}
Note from \autoref{sec:auxseries} that the elements
  \[
    1, \quad (\eta(y_1) - \gamma(y_1))(1 - 4y_1)^{\frac{1}{2}}, \quad \eta_1(y_1)(1 - 4y_1)^{\frac{1}{2}}, \quad \eta_2(y_1)(1 - 4y_1)^{\frac{1}{2}}, \quad \ldots
  \]
  are polynomials in $(1 - 4y_1)^{-1}$ of degree $0, 1, 2, 3, \ldots$ respectively, so by \autoref{thm:delta_H}, we know that we can write
  \begin{equation}\label{eq:linear-combination}
    (1-\eta)^{2g-1}\Delta_1 \Mon_g = F_{g,0} (1 - 4y_1)^{-\frac{1}{2}} + F_{g,1} \big( \eta(y_1) - \gamma(y_1) \big) + \sum_{j = 2}^{3g - 1} F_{g,j} \eta_{j-1}(y_1),
  \end{equation}
  where, for $j = 0, 1, \ldots, 3g - 1$, $F_{g,j}$ is an element of $\QQ[\eta_k (1 - \eta)^{-1}]_{k \geq 1}$. Note also that $F_{g,j}$ has weighted degree at most $3g-1-j$, for $j=0,\ldots ,3g-1$. If we set $y_1 = 0$ in \eqref{eq:linear-combination}, we get 
  \begin{equation}\label{eq:cond1}
  F_{g,0} = 0,
  \end{equation}
  since $\Delta_1 \Mon_g$ has no constant term as a power series in $y_1$. 
  
  Next, note that when we are dealing with polynomials in $(1 - 4y_1)^{-1}$, we can evaluate them at $y_1 = \infty$, or equivalently, at $(1 - 4y_1)^{-1} = 0$, and denote this evaluation by the operator $\Omega$. Now suppose we apply the operator
  \[
   \Omega (1 - \T) (1 - 4y_1)^{\frac{1}{2}}
  \]
  to~\eqref{eq:linear-combination}. Taking into account~\eqref{eq:cond1}, we obtain on the right-hand side
  \begin{equation}\label{eq:rhsOMT}
    F_{g,1} \Omega (1 - \T) (1-4y_1)^{\frac{1}{2}}( \eta(y_1) - \gamma(y_1)) + \sum_{j = 2}^{3g - 1} F_{g,j}\Omega(1 - \T) (1-4y_1)^{\frac{1}{2}} \eta_{j-1}(y_1).
  \end{equation}
  By direct computation, we have
  \[
    \Omega (1 - \T) (1-4y_1)^{\frac{1}{2}}( \eta(y_1) - \gamma(y_1)) = -1.
  \]
  To evaluate the remaining terms in~\eqref{eq:rhsOMT}, note that for $j \geq 2$, we have $(1-4y_1)^{\frac{3}{2}}\eta_{j-1}(y_1)=y_1 a_{j-1}(y_1)$, where from \autoref{sec:auxseries}, we know that $a_{j-1}(y_1)$ is a polynomial in $(1-4y_1)^{-1}$ with no constant term (in $(1-4y_1)^{-1}$). It follows that
  \[
    \Omega (1-4y_1)^{\frac{1}{2}} \eta_{j-1}(y_1) = 0,
  \]
  and it is routine to check that
  \begin{equation*}
    \Split_{1 \to 2} \left( (1 - 4y_1)^{\frac{3}{2}}\eta_{j-1}(y_1)\right)
    = \left(\frac{1}{ 1-4y_1}-1\right) \frac{y_2}{1-4y_2}\frac{a_{j-1}(y_1)-a_{j-1}(y_2)}{(1-4y_1)^{-1}-(1-4y_2)^{-1}},
  \end{equation*}
  so that we have
  \[
    \Omega\T (1-4y_1)^{\frac{1}{2}}\eta_{j-1}(y_1) = - (1-\eta)^{-1}\Pi_2 y_2(1-4y_2)^{-\frac{3}{2}}a_{j-1}(y_2)=-\frac{\eta_{j-1}}{1-\eta},
  \]
  from~\eqref{eq:PiAux}. Putting these together,~\eqref{eq:rhsOMT} becomes
  \[
    -F_{g,1} + \sum_{j=2}^{3g-1} F_{g,j}\frac{\eta_{j-1}}{1-\eta}.
  \]

Now, when we apply $\Omega (1 - \T) (1 - 4y_1)^{\frac{1}{2}}$ to the left-hand side of~\eqref{eq:linear-combination}, and use~\eqref{eq:revjcut}, we get
\[
(1-\eta)^{2g-2}\Omega\left( \Delta_1^2 \Mon_{g-1} + \sum_{g'=1}^{g-1} \Delta_1 \Mon_{g'} \, \Delta_1 \Mon_{g-g'}\right) .
\]
But from the proof of \autoref{thm:delta_H}, specifically the analysis of the right-hand side of~\eqref{eq:indHg}, we see that for $g\geq 2$, every term in the summation over $g'$ is a polynomial in $(1-4y_1)^{-1}$ multiplied by an additional factor of $(1-4y_1)^{-1}$, and so $\Omega$ sends the summation to $0$. We can also deduce from \autoref{thm:delta_values} that the remaining term is also sent to $0$ by $\Omega$. Putting both sides together and multiplying by $2\gamma(1-\eta)/(1-\gamma)$, we obtain the equation
\begin{equation}\label{eq:cond2}
 -F_{g,1}\frac{2\gamma(1-\eta)}{1-\gamma}+\sum_{j=2}^{3g-1}F_{g,j}\frac{2\gamma\eta_{j-1}}{1-\gamma}=0.
\end{equation}
 
 Now, from~\eqref{eq:linear-combination}, using~\eqref{eq:cond1} and~\eqref{eq:propPhi}, we have
 \begin{equation*}
 \begin{aligned}
 (1-\eta)^{2g-1}\Phi\Delta_1\Mon_g&=
 F_{g,1}\frac{2\gamma(1-\eta)}{1-\gamma}+F_{g,2}\left( \eta-\frac{2\gamma\eta_1}{1-\gamma} \right) \\
 &\qquad\qquad\quad\;\;\qquad +\sum_{j=3}^{3g-1} F_{g,j}\left(\eta_{j-2} -\frac{2\gamma\eta_{j-1}}{1-\gamma}\right)\\
 &= F_{g,2}\eta+\sum_{j=3}^{3g-1} F_{g,j}\eta_{j-2},
 \end{aligned}
 \end{equation*}
where the second equality follows from~\eqref{eq:cond2}. Thus, from~\eqref{eq:invertDel}, we have
\begin{equation*}
\Mon_g=\int\limits_{0}^{1}\frac{\mathrm{d}t}{t}\Theta_t \left( \frac{F_{g,2}\eta+\sum_{j=3}^{3g-1} F_{g,j}\eta_{j-2}}{(1-\eta)^{2g-1}}\right) .
\end{equation*}

But $F_{g,j}$ has weighted degree at most $3g-1-j$, and using~\eqref{eq:propThetaseries} we obtain
\begin{equation}\label{eq:Monfinalint}
\begin{aligned}
\Mon_g &= \int\limits_{0}^{1}\mathrm{d}t\Bigg( \sum_{d=0}^{3g-3}\sum_{\alpha\vdash d} b_{g,\alpha}\eta_{\alpha}\eta\frac{t^{\ell(\alpha)}}{(1-\eta t)^{\ell(\alpha)+2g-1}}\\
&\qquad\qquad + \sum_{d=1}^{3g-3}\sum_{\alpha\vdash d} e_{g,\alpha}\eta_{\alpha}\frac{t^{\ell(\alpha)-1}}{(1-\eta t)^{\ell(\alpha)+2g-2}}\Bigg) ,
\end{aligned}
\end{equation}
where $b_{g,\alpha}$ and $e_{g,\alpha}$ are rational numbers. Now, integrating, we obtain
\begin{equation*}
\begin{aligned}
\int\limits_0^1 \frac{t^m}{(1-\eta t)^{m+2g-1}}\mathrm{d}t &=\frac{1}{\eta^{m+1}}\int\limits_0^{\frac{\eta}{1-\eta}}z^m(1+z)^{2g-3}\mathrm{d}z,\qquad\text{where $z=\frac{\eta t}{1-\eta t}$},\\
&= \sum_{i=0}^{2g-3}\binom{2g-3}{i}\frac{1}{m+1+ i}\frac{\eta^{i}}{(1-\eta)^{m+1+ i}},
\end{aligned}
\end{equation*}
which is equal to $(1-\eta)^{m+1}$ times a polynomial over $\QQ$ in $(1-\eta)^{-1}$ of degree at most $2g-3$. The result follows by applying this to each term of~\eqref{eq:Monfinalint}, using $\eta=1-(1-\eta)$ for the isolated $\eta$ in the first summation.
\end{proof}

Finally, by refining the polynomiality result of \autoref{thm:MonTwoPoly}, we are able to prove the explicit form for the monotone Hurwitz number generating function with genus $g\geq 2$ given in \thmpartref{thm:main}{item:rational-form}.

\begin{theorem}\label{thm:Mongenfcntwo}
  For $g \geq 2$, the generating function for genus $g$ monotone single Hurwitz numbers is given by
  \[
    \Mon_g = -c_{g,(0)} + \sum_{d = 0}^{3g - 3} \sum_{\alpha \vdash d} \frac{c_{g,\alpha} \, \eta_\alpha}{(1 - \eta)^{\ell(\alpha) + 2g-2}},
  \]
  where the $c_{g,\alpha}$ are rational numbers.
\end{theorem}

\begin{proof}
From \autoref{thm:MonTwoPoly}, we know that $\Mon_g$ is a linear combination of the monomials $\rho_{\alpha,n}= \eta_{\alpha}(1-\eta)^{-\ell(\alpha)-n}$, where  $\abs{\alpha} =d\leq 3g-3$ and $n\leq 2g-2$. Then $\Delta_1\rho_{(0),0}=0$, and from \thmpartref{prop:DelR}{item:DelR-i}, we have
\[
(1-\eta)^{n+1}(1-4y_1)^{\frac{1}{2}}\Delta_1\rho_{\alpha,n}\in\mcR_{d+2}.
\]
 Then, if $(\alpha,n)\neq((0),0)$, \autoref{thm:delta_H} implies that $n=2g-2$, so we have
 \[
\Mon_g=c+ \sum_{d = 0}^{3g - 3} \sum_{\alpha \vdash d} \frac{c_{g,\alpha} \, \eta_\alpha}{(1 - \eta)^{ \ell(\alpha)+2g-2}}, \]
where $c_{g,\alpha}$ are rational numbers. But $\Mon_g$ has constant term $0$, so $c=-c_{g,(0)}$, giving the result.
\end{proof}
%----------------------------------------------------------------
\section{Explicit formulas for monotone Hurwitz numbers}\label{sec:explicit}
%----------------------------------------------------------------

In the previous section, we obtained explicit results for $\Mon_g$, $g\geq 1$. In this section, we consider the coefficients in these generating functions. To begin,
the coefficient extraction operators $[p_\alpha]$ and $[q_\alpha]$, defined on $\QQ[[\pp]] = \QQ[[\qq]]$, can be expressed in terms of each other using the multivariate Lagrange Implicit Function Theorem~\cite[Theorem 1.2.9]{goulden-jackson04}, as follows.

\begin{lemma}\label{thm:LIFT}
  If $\alpha \vdash d$ is a partition and $F$ is an element of $\QQ[[\pp]]$, then
  \[
    [p_\alpha] F = [q_\alpha] \frac{(1 - \eta) F}{(1 - \gamma)^{2d + 1}},
  \]
  where
  \[
    q_j = p_j (1-\gamma )^{-2j}, \quad j \geq 1,\qquad \gamma = \sum_{k \geq 1} \binom{2k}{k} q_k, \qquad \eta = \sum_{k \geq 1} (2k + 1) \binom{2k}{k} q_k.
  \]
\end{lemma}

\begin{proof}
  Let $\phi_j = (1 - \gamma)^{-2j}$, so that $q_j = p_j \phi_j$, $ j\geq 1$. Then, from the multivariate Lagrange Implicit Function Theorem~\cite[Theorem 1.2.9]{goulden-jackson04}, we have
  \begin{align*}
    [p_\alpha] F
      &= [q_\alpha] F \, \phi_\alpha \det\left( \delta_{ij} - q_j \diff{q_j} \log \phi_i \right)_{i,j \geq 1} \\
      &= [q_\alpha] F \, \phi_\alpha \det\left( \delta_{ij} - \frac{2i q_j}{1 - \gamma} \binom{2j}{j} \right)_{i,j \geq 1},
  \end{align*}
  where $\phi_\alpha = \prod_{j \geq 1} \phi_{\alpha_j}$. We have $\phi_\alpha = (1 - \gamma)^{-2d}$, and using the fact that $\det(I + M) = 1 + \trace(M)$ for any matrix $M$ of rank zero or one, we can evaluate the determinant as
  \[
    \det\left( \delta_{ij} - q_j \diff{q_j} \log \phi_i \right)_{i,j \geq 1}
      = 1 - \sum_{k \geq 1} \frac{2k q_k}{1 - \gamma} \binom{2k}{k}
      = \frac{1 - \eta}{1 - \gamma}.
  \]
  Substituting, we obtain
  \[
    [p_\alpha] F = [q_\alpha] \frac{(1 - \eta) F}{(1 - \gamma)^{2d + 1}}. \qedhere
  \]
\end{proof}

Using Lemma~\ref{thm:LIFT}, we are now able to obtain the explicit formula given in \autoref{thm:Mongenusone} for the genus one monotone Hurwitz numbers $\mon_1(\alpha)$.

\begin{theorem}
 The genus one monotone single Hurwitz numbers $\mon_1(\alpha)$, $\alpha \vdash d$ are given by
  \begin{multline*}
    \mon_1(\alpha) = \frac{1}{24} \frac{d\,!}{\abs{\Aut \alpha}} \prod_{i=1}^{\ell(\alpha)} \binom{2\alpha_i}{\alpha_i} \\
    \times \left( (2d + 1)^{\overline{\ell(\alpha)}} - 3 (2d + 1)^{\overline{\ell(\alpha) - 1}} - \sum_{k = 2}^{\ell(\alpha)} (k - 2)! (2d + 1)^{\overline{\ell(\alpha) - k}} e_k(2\alpha + 1) \right).
  \end{multline*}
\end{theorem}

\begin{proof} From Theorem~\ref{cor:genfcnone}, we have
\begin{equation}\label{eq:explicit_genfunc}
  \Mon_1 = \sum_{d \geq 1} \sum_{\alpha \vdash d} \mon_1(\alpha) \frac{p_\alpha}{d\,!} = \tfrac{1}{24} \log\frac{1}{1 - \eta} - \tfrac{1}{8} \log\frac{1}{1 - \gamma}.
\end{equation}
For the first term in $\Mon_1$, applying \autoref{thm:LIFT}, we obtain
\begin{align*}
  [p_\alpha] \log \frac{1}{1 - \eta} 
    &= [q_\alpha] \frac{1 - \eta}{(1 - \gamma)^{2d + 1}} \log\frac{1}{1 - \eta} \\
    &= [q_\alpha] \left( \sum_{j \geq 0} (2d + 1)^{\overline{j}} \frac{\gamma^j}{j!} \right) \left( \eta - \sum_{k \geq 2} (k - 2)! \frac{\eta^k}{k!} \right) \\
    &= [q_\alpha] \left( (2d + 1)^{\overline{\ell - 1}} \frac{\gamma^{\ell - 1} \eta}{(\ell - 1)!} - \sum_{k = 2}^{\ell} (k - 2)! (2d + 1)^{\overline{\ell - k}} \frac{\gamma^{\ell - k}}{(\ell - k)!} \frac{\eta^k}{k!} \right).
\end{align*}
For the remaining term in $\Mon_1$, we apply \autoref{thm:LIFT} again, together with equation~\eqref{eq:dieqdiep}, to obtain
\begin{align*}
  [p_\alpha] \log\frac{1}{1 - \gamma}
    &= \tfrac{1}{d} [p_\alpha] \DD \log\frac{1}{1 - \gamma} 
%    &= \tfrac{1}{d} [q_\alpha] \frac{1 - \eta}{(1 - \gamma)^{2d + 1}} \DD \log\frac{1}{1 - \gamma} \\
%    &= \tfrac{1}{d} [q_\alpha] \frac{1}{(1 - \gamma)^{2d}} \EE \log\frac{1}{1 - \gamma} \\
    = \tfrac{1}{d} [q_\alpha] \EE \left( \frac{1}{2d (1 - \gamma)^{2d}} \right) \\
    &= [q_\alpha] \left( \frac{1}{2d (1 - \gamma)^{2d}} \right) 
    = [q_\alpha] (2d + 1)^{\overline{\ell(\alpha) - 1}} \frac{\gamma^{\ell(\alpha)}}{\ell(\alpha)!}.
\end{align*}
But iterating the product rule gives
\begin{equation*}
\begin{aligned}
 & \abs{\Aut\alpha} [q_\alpha] \frac{\gamma^{\ell(\alpha) - k}}{(\ell(\alpha) - k)!} \frac{\eta^k}{k!} = \frac{\partial^{\ell(\alpha)}}{\partial q_\alpha} \left( \frac{\eta^k}{k!} \frac{\gamma^{\ell(\alpha) - k}}{(\ell(\alpha) - k)!} \right)\\
    &\;\; = \prod_{i = 1}^{\ell(\alpha)} \binom{2\alpha_i}{\alpha_i} \sum_{1 \leq i_1 < \cdots < i_k \leq \ell(\alpha)} (2\alpha_{i_1} + 1) (2\alpha_{i_2} + 1) \cdots (2\alpha_{i_k} + 1)  \\
    &\;\; = \prod_{i = 1}^{\ell(\alpha)} \binom{2\alpha_i}{\alpha_i} e_k(2\alpha + 1).
  \end{aligned}
\end{equation*}
The explicit expression for $\mon_1(\alpha)$ follows by combining the above results, and using the facts that $e_0(\alpha) = 1$ and $e_1(\alpha)=2d+\ell(\alpha)$.
\end{proof}

Finally, we prove the polynomiality result for monotone Hurwitz numbers stated in \autoref{thm:polynomiality}.

\begin{theorem}\label{thm:polynomiality1}
  For each pair $(g,\ell)$ with $(g,\ell) \notin \{(0,1), (0,2)\}$, there is a polynomial $\vec{P}_{g,\ell}$ in $\ell$ variables such that, for all partitions $\alpha \vdash d$ with $\ell$ parts,
  \[
    \mon_g(\alpha) = \frac{d\,!}{\abs{\Aut \alpha}} \vec{P}_{g,\ell}(\alpha_1, \dots, \alpha_\ell) \; \prod_{j=1}^\ell \binom{2\alpha_j}{\alpha_j}.
  \]
\end{theorem}

\begin{proof}
  For $g = 0$, this follows from the explicit formula for genus zero monotone Hurwitz numbers given in~\cite{goulden-guay-paquet-novak12a}, which has this form for $\ell \geq 3$. For $g \geq 1$, by applying \autoref{thm:LIFT}, we obtain
  \[
    \mon_g(\alpha) = d\,! [p_\alpha] \Mon_g = d\,! [q_\alpha] \frac{(1 - \eta) \Mon_g}{(1 - \gamma)^{2d+1}}.
  \]
  Given the general form from \autoref{thm:Mongenfcntwo}, the power series on the right-hand side can be expanded as an infinite sum of (rational multiples of) terms of the form
  \[
    \binom{-2d-1}{m} \gamma^m \eta^{n_0} \eta_1^{n_1} \eta_2^{n_2} \cdots \eta_k^{n_k},
  \]
  where $m, n_0, n_1, \ldots, n_k \geq 0$ are integers. However, since the series $\gamma, \eta, \eta_1, \eta_2, \ldots$ are all linear in the indeterminates $\qq$, only the finitely many terms with $m + n_0 + n_1 + \cdots + n_k = \ell$ contribute to the coefficient of $q_\alpha$. For $m$ fixed, the binomial coefficient $\binom{-2d-1}{m}$ is a polynomial in the parts of $\alpha$, and given the definition of the series $\gamma, \eta, \eta_1, \eta_2, \ldots$, the contribution to the coefficient of $q_\alpha$ is a polynomial in the parts of $\alpha$ multiplied by the factor
  \[
    \frac{1}{\abs{\Aut \alpha}} \prod_{j=1}^\ell \binom{2\alpha_j}{\alpha_j}.
  \]
  It follows that $\mon_g(\alpha)$ has the stated form.
\end{proof}

\appendix
%----------------------------------------------------------------
\section{Rational forms for genus three}\label{sec:forms}
%----------------------------------------------------------------

The following equation gives the rational form for the genus three generating series for the monotone single Hurwitz numbers, as described in \autoref{thm:Mongenfcntwo}:

\vspace{.05in}

\begin{align*}
2^{-2}\cdot 9!\, \vec{\mathbf{H}}_3 = \Big( 90 &+ \frac{-90}{(1 - \eta)^4}\Big)
+ \frac{70 \eta_6 + 63 \eta_5 - 377 \eta_4 - 189 \eta_3 + 667 \eta_2 + 126 \eta_1}{(1 - \eta)^5} \\
&+ \frac{1078 \eta_1 \eta_5 + 2012 \eta_2 \eta_4 + 1214 \eta_3^2 + 1209 \eta_1 \eta_4 }{(1 - \eta)^6} \\
&+ \frac{1998 \eta_2 \eta_3 - 3914 \eta_1 \eta_3 - 2627 \eta_2^2 - 2577 \eta_1 \eta_2 + 1967 \eta_1^2}{(1 - \eta)^6} \\
&+ \frac{8568 \eta_1^2 \eta_4 + 26904 \eta_1 \eta_2 \eta_3  + 5830 \eta_2^3 + 10092 \eta_1^2 \eta_3}{(1 - \eta)^7} \\
&+ \frac{13440 \eta_1 \eta_2^2 - 20322 \eta_1^2 \eta_2 - 4352 \eta_1^3}{(1 - \eta)^7} \\
&+ \frac{ 44520 \eta_1^3 \eta_3 + 86100 \eta_1^2 \eta_2^2 +  49980 \eta_1^3 \eta_2 - 15750 \eta_1^4}{(1 - \eta)^8} \\
&+ \frac{162120 \eta_1^4 \eta_2 + 31080 \eta_1^5}{(1 - \eta)^9}
+ \frac{68600 \eta_1^6}{(1 - \eta)^{10}}
\end{align*}

\vspace{.05in}

This should be compared with the genus three generating series for the single Hurwitz numbers that appeared in~\cite{goulden-jackson-vakil01}:

\vspace{.05in}

\begin{align*}
2^4\cdot 9!\, \mathbf{H}_3 = 
&{}\quad \frac{70 \phi_6 - 294 \phi_5 + 410 \phi_4 - 186 \phi_3 }{(1 - \phi)^5} \\
&+ \frac{1078 \phi_1 \phi_5 + 2012 \phi_2 \phi_4 + 1214 \phi_3^2 + 2418 \phi_1 \phi_4 }{(1 - \phi)^6} \\
&+ \frac{- 6156 \phi_2 \phi_3 + 4658 \phi_1 \phi_3 + 3002 \phi_2^2 - 1860 \phi_1 \phi_2}{(1 - \eta)^6} \\
&+ \frac{8568 \phi_1^2 \phi_4 + 26904 \phi_1 \phi_2 \phi_3 + 5830 \phi_2^3 - 25968 \phi_1^2 \phi_3 }{(1 - \phi)^7} \\
&+ \frac{ - 33642 \phi_1 \phi_2^2 + 25770 \phi_1^2 \phi_2 - 2790\phi_1^3}{(1 - \phi)^7} \\
&+ \frac{44520 \phi_1^3 \phi_3 + 86100 \phi_1^2 \phi_2^2 - 110600 \phi_1^3 \phi_2 + 21420 \phi_1^4}{(1 - \phi)^8} \\
&+ \frac{162120 \phi_1^4 \phi_2 - 62440 \phi_1^5}{(1 - \phi)^9}
+ \frac{68600 \phi_1^6}{(1 - \phi)^{10}}
\end{align*}

\vspace{.05in}

\raggedbottom
\pagebreak

\bibliographystyle{myplain}
\bibliography{monotone}
\end{document}